\documentclass[preprint]{imsart}
\arxiv{1703.01332}

\usepackage[numbers]{natbib}
\usepackage{hyperref}
\usepackage{cleveref}
\usepackage{enumitem}
\usepackage[T1]{fontenc}
\usepackage{lmodern}
\usepackage{graphicx}
\usepackage{algorithm2e}
\usepackage{amssymb,amsmath,amsthm}

\usepackage{natbib}
\usepackage{hyperref}
\usepackage{cleveref}
\usepackage{autonum}
\usepackage{mathtools}
\usepackage{thmtools}
\usepackage[top=1.5in, bottom=1.5in, left=1.7in, right=1.7in]{geometry}

\usepackage{booktabs} 

\usepackage{algorithmicx}
\hypersetup{breaklinks=true,
    linkcolor=blue,
    citecolor=magenta,
    colorlinks=true,
pdfborder={0 0 0}}

\usepackage{amsthm}

\numberwithin{equation}{section}
\startlocaldefs

\crefname{theorem}{theorem}{theorems}
\crefname{lemma}{lemma}{lemmas}
\crefname{proposition}{proposition}{propositions}
\crefname{assumption}{assumption}{assumptions}
\crefname{example}{example}{examples}
\crefname{corollary}{corollary}{corollaries}

\declaretheorem[name=Theorem,numberwithin=section]{theorem}
\declaretheorem[name=Proposition,sibling=theorem]{proposition}
\declaretheorem[name=Lemma,sibling=theorem]{lemma}

\declaretheorem[name=Assumption]{assumption}

\declaretheorem[name=Definition,style=definition]{definition}

\numberwithin{equation}{section}
\numberwithin{theorem}{section}

\usepackage{enumitem}

\usepackage{times}

\newcommand{\design}{\mathbf{X}}
\newcommand{\R}{\mathbf{R}}
\newcommand{\E}{\mathbb{E}} 
\newcommand{\eps}{\varepsilon}
\newcommand{\hbeta}{\hat\beta}
\DeclareMathOperator*{\argmin}{argmin}

\DeclareMathOperator*{\supp}{supp}
\DeclareMathOperator*{\dom}{dom}
\DeclareMathOperator*{\Tr}{Trace}

\usepackage{etoolbox}
\newbool{expandrisk}
\booltrue{expandrisk}

\newcommand{\risk}{
    \ifbool{expandrisk}{
        \|\design(\hbeta-\beta^*) \|
    }{\hat r}
}

\begin{document}
\title{Optimistic lower bounds for convex regularized least-squares}
\author{Pierre C. Bellec}
\address{
 Department of Statistics and Biostatistics \\
 Rutgers, The State University of New Jersey \\
 }

\begin{abstract}
    Minimax lower bounds are pessimistic in nature:
    for any given estimator, minimax lower bounds
    yield the existence of a worst-case target vector
    $\beta^*_{worst}$ for which the prediction error
    of the given estimator is bounded from below.
    However, minimax lower bounds shed no light
    on the prediction error of the given estimator
    for target vectors different than $\beta^*_{worst}$.

    A characterization of the prediction error of any convex regularized least-squares is given.
    This characterization provide both
    a lower bound and an upper bound on the prediction error.
    This produces lower bounds that are applicable
    for any target vector and not only for a single, worst-case $\beta^*_{worst}$.

    Finally, these lower and upper bounds on the prediction error are applied to the Lasso
    in sparse linear regression. We obtain a lower bound involving the compatibility constant for any tuning parameter,
    matching upper and lower bounds for the universal choice of the tuning parameter,
    and a lower bound for the Lasso with small tuning parameters.
\end{abstract}

\maketitle

\section{Introduction}

We study the linear regression problem 
\begin{equation}
    y = \design \beta^* + \eps,
    \label{linear-model}
\end{equation}
where one observes $y\in\R^n$, the design matrix $\design\in\R^{n\times p}$ is known and deterministic
and $\eps$ is a noise random vector.
The prediction error of an estimator $\hbeta$
is given by 
\begin{equation}
    \|\design(\hbeta - \beta^*)\|,
\end{equation}
where $\|\cdot\|$ is the Euclidean norm in $\R^n$.
This paper provides a characterization of the prediction error of convex regularized estimators,
that is, estimators $\hbeta$ that solve the minimization problem
\begin{equation}
    \label{hbeta}
    \hbeta\in\argmin_{\beta\in\R^p}
    \|\design\beta-y\|^2 + 2 h(\beta),
\end{equation}
where $h:\R^p\rightarrow [0,+\infty]$ a convex penalty that satisfies the following assumption.

\begin{assumption}
    \label{assum-h}
    The penalty function $h:\R^p\rightarrow [0,+\infty]$ is convex,
    proper and such that the minimization problem \eqref{hbeta} has at least one solution
    for any $y\in\R^n$ and $\design\in\R^{n\times p}$.
\end{assumption}

Recall that $h:\R^p\rightarrow [0,+\infty]$ is proper if $h(x) < +\infty$ for at least one $x\in\R^p$.
\Cref{assum-h} is satisfied for any sensible penalty function $h$.
Since a convex function $\R^p\rightarrow [0,+\infty]$ has at least global minimizer
provided that it is proper, lower-semicontinuous and coercive \cite[Theorem 2.19]{peypouquet2015convex},
the following examples satisfy \Cref{assum-h}.

\begin{itemize}
    \item $h(\cdot) = \lambda N(\cdot)^q$ for any norm $N(\cdot)$, tuning parameter $\lambda>0$ and integer $q\ge1$.
        For instance, $h(\cdot) = \lambda \|\cdot\|_1$ corresponds to the Lasso penalty
        and $h(\cdot)= \lambda \|\cdot\|^2$ corresponds to Ridge regression.
    \item $h(\cdot) = \delta_K(\cdot)$ where $\delta_K$ is the indicator function of a nonempty closed convex set $K\subset\R^p$,
        that is,
        $\delta_K(x) = +\infty$ if $x\notin K$ and $\delta_K(x) = 0$ if $x\in K$.
    \item $h(\cdot) = g(\cdot)+ \delta_K(\cdot)$ where $g$ is a finite convex function and $K$ is a nonempty closed convex set.
\end{itemize}

This paper studies the prediction error $\risk$
of convex regularized least-squares, i.e., solutions of the minimization problem \eqref{hbeta}.
A common paradigm in theoretical statistics or machine learning
is the minimax framework. In the minimax framework,
the goal is to construct estimators that have the smallest possible prediction error,
uniformly over a class of target vectors.
In this minimax framework, lower bounds are usually obtained using information theoretic tools
such as Le Cam's Lemma or Fano's inequality, see for instance \cite{yu1997assouad}
or Section 2 in \cite{tsybakov2009introduction}.
These minimax lower bounds are \emph{pessimistic} in nature:
for any given estimator, minimax lower bounds yield the existence of a worst-case target vector $\beta^*_{worst}$
for which the prediction error of the given estimator is bounded from below.
However, minimax lower bounds shed no light on the prediction error of a given estimator
for target vectors that are not equal to $\beta^*_{worst}$.

The main goal of this paper is to propose a machinery to derive lower bounds
on the prediction error of a convex regularized least-squares \eqref{hbeta}.
This machinery yields lower bounds on the prediction error for any target vector
and not only for a single, worst-case $\beta^*_{worst}$.
Because of this contrast with minimax lower bounds,
we coined the lower bounds of the present paper \emph{optimistic}.

The paper is organized as follows.
The next section defines the functions $F, G$ and $H$ that will be used to charactize
the prediction error of convex regularized least-squares \eqref{hbeta}.
\Cref{s:optimistic} proposes several lower-bound
results based on the functions $F, G, H$.
In \Cref{s:lasso}, we apply these lower-bound results
to sparse linear regression.
We will see that the optimistic lower bounds of the present paper
shed light on the performance of the Lasso, which is the estimator \eqref{hbeta}
with penalty $h(\cdot)$ proportional to the $\ell_1$-norm.
Finally, \Cref{s:gaussian} study concentration properties of the prediction
error $\risk$ when the noise $\eps$ has standard normal distribution.

\section{Variational characterizations of the prediction error}
\label{s:variational-characterization}

Define the function
$F:\R\rightarrow [-\infty,+\infty)$
by
\begin{align}
    F(t)
    &\coloneqq
        \sup_{\beta\in\R^p: \|\design(\beta-\beta^*)\| \le t}
        \left(
        \eps^T\design(\beta-\beta^*)
        - h(\beta)
        \right)
    - t^2/2,
    \label{def-F}
\end{align}
for all $t\ge 0$ and all realizations $\eps\in\R^n$ of the random vector,
with the convention that the supremum over an empty set is equal to $-\infty$.
As the function $F$ depends on the noise random vector $\eps$,
the function $F$ is random in the sense that for all $t$,
$F(t)$ is a random variable valued in $[-\infty, \infty)$.
The following proposition is the starting point of the results of this note.

\begin{proposition}
    \label{prop:maximizer-of-F}
    Let $h:\R^p\rightarrow [0,+\infty]$
    be any function
    and assume that there exists a solution $\hbeta$ 
    to the minimization problem \eqref{hbeta}.
    Then for all $t\ge 0$ we have
    \begin{equation}
        F(\risk)
        \ge
        F(t),
        \label{comparison-F}
    \end{equation}
    that is, the prediction error $\|\design(\hbeta-\beta^*)\|$ is a maximizer of the function
    $F(\cdot)$ for any realization of the noise vector $\epsilon\in\R^n$.
\end{proposition}
The above proposition shows that the prediction error $\risk$
is a maximizer of $F$ for any penalty function $h$.
This observation was initially made in the context of shape restricted regression by \cite{chatterjee2014new}.
With the notation of the present paper, \cite{chatterjee2014new} considers penalty functions $h$
that are indicator functions of closed convex sets.
Proposition~\ref{prop:maximizer-of-F} extends the initial observation of \cite{chatterjee2014new} to any penalized estimator.
Such extension was also proposed in \cite{chen2017note} concurrently and
contemporaneously of the present note.

Variational characterization of functionals of the estimator $\hbeta$ have been also studied in the following works.
\cite[Theorem 3.1]{bartlett2006empirical} and
\cite[Seciton 3]{saumard2012optimal} show that the excess risk in empirical risk minimization
can be essentially characterized as the maximizer of some objective function.
For penalty $h(\cdot)$ of the form
$h(\cdot) = I^2(\cdot)$ for some seminorm $I$,
\cite{muro2015concentration}
study the quantity $\tau(\hbeta) = \risk+h(\hbeta)$ and prove that this quantity sharply concentrates around a point that can be characterized
as the maximizer some deterministic objective function.
\cite{vandegeer2015concentration} study the same function $\tau(\cdot)$ for more general estimators
that include maximum likelihood estimators for generalized linear models.
\cite{saumard2017new}
and
\cite[Section 5.2]{navarro2015slope} use another variational characterization to study the infinity norm of $\hbeta - \beta^*$.

\begin{proof}[Proof of \Cref{prop:maximizer-of-F}]
    Let $\hat r=\risk$ for brevity.
    \boolfalse{expandrisk}
    For all $\beta\in\R^p$, inequality
    $\|\design\beta - y\|^2 + 2h(\beta) \ge \|\design\hbeta - y\|^2 + 2h(\hbeta)\|$
    can be rewritten as
    \begin{equation}
        \eps^T\design(\beta-\beta^*) - h(\beta) - \|\design(\beta - \beta^*)\|^2/2
        \le
        \eps^T\design(\hbeta-\beta^*) - h(\hbeta) - \risk^2/2.
    \end{equation}
    This inequality implies that the right hand side of the previous display is equal to $F(\risk)$.
    Now let $\beta$ be such that $\|\design(\beta-\beta^*)\|\le t$. Then we have
    \begin{align}
        \eps^T\design(\beta-\beta^*) - h(\beta) - t^2/2 
        &\le \eps^T\design(\beta-\beta^*) - h(\beta) - \|\design(\beta - \beta^*)\|^2/2, \\
        &\le \eps^T\design(\hbeta-\beta^*) - h(\hbeta) - \risk^2/2, \\
        &=
        F(\risk).
    \end{align}
    By definition of the supremum, we have established \eqref{comparison-F}.
    \booltrue{expandrisk}
\end{proof}

Interestingly, the function $h$ need not be convex in Proposition~\ref{prop:maximizer-of-F}:
the above result holds as long as a solution to the minimization problem \eqref{hbeta} exists.
In the next results, the function $h$ is assumed to be convex.

A function $u:\R\rightarrow [-\infty, +\infty)$ is said to be $\gamma$-strongly concave if and only if
the function $t\mapsto u(t)+\gamma t^2/2$ is concave on $\R$.
If the penalty function $h$ is convex, then we have the following.

\begin{proposition}
    \label{prop:strong-concavity-F}
    If the penalty function $h(\cdot)$ is convex
    then the function $F(\cdot)$ is 1-strongly concave.
\end{proposition}
\begin{proof}
    Define the function $M(\cdot)$ by 
    $M(t) = F(t) +t^2/2$ for all $t\in\R$.
    It is enough to prove that $M$ is concave.
    Let $\alpha\in[0,1]$.
    Let $\beta_s,\beta_t \in \dom h$ be such that 
    $\|\design(\beta_s-\beta^*)\|\le s$ and
    $\|\design(\beta_t-\beta^*)\|\le t$.
    Let $\beta=\alpha \beta_t + (1-\alpha)\beta_s$.
    By the triangle inequality, 
    $\|\design(\beta-\beta^*)\| \le \alpha t + (1-\alpha)s$.
    Furthermore, by convexity of $h$ we have $\beta\in\dom h$
    and $h(\alpha \beta_t + (1-\alpha) \beta_s ) \le \alpha h(\beta_t) + (1-\alpha) h(\beta_s)$.
    Thus
    \begin{multline}
        \alpha
        \left[
            \eps^T\design(\beta_t-\beta^*)
            - h(\beta_t)
        \right]
        +
        (1-\alpha)
        \left[
            \eps^T\design(\beta_s-\beta^*)
            - h(\beta_s)
        \right] \\
        \le
        \eps^T\design(\beta-\beta^*) - h(\beta)
        \le M(\alpha t + (1-\alpha) s).
    \end{multline}
    By definition of the supremum, we have established that 
    \begin{equation}
        \alpha M(t) + (1-\alpha) M(s) \le M(\alpha t +(1-\alpha) s)
        \label{concavity-M}
    \end{equation}
    provided that both $M(t)$ and $M(s)$ are not $-\infty$.
    If $M(t)$ or $M(s)$ is equal to $-\infty$, then \eqref{concavity-M} trivially holds.
    This proves that $M(\cdot)$ is concave, 
    and since $M(t) = F(t) + t^2/2$, this also proves that $F$ is 1-strongly concave.
\end{proof}
By Proposition~\ref{prop:maximizer-of-F}, the quantity $\risk$
is a maximizer of $F$ as long as there exists a solution to \eqref{hbeta}.
Since a strongly concave function admits at most one maximizer, $\risk$ is the
only maximizer of $F$ provided that the penalty $h$ is convex.
The next proposition introduces the function $G$ which is also maximized at $\risk$.

\begin{proposition}
    \label{prop:G}
    Let Assumption~\ref{assum-h} be fulfilled
    and define the function $G$ by
    \begin{equation}
        G(t) \coloneqq \sup_{\beta\in\R^p:\|\design(\beta-\beta^*)\|\le t } \left(
            \eps^T\design(\beta-\beta^*) - h(\beta)
        \right)
        - t \; \|\design(\hbeta-\beta^*) \|.
        \label{def-G}
    \end{equation}
    Then $G$ is concave and the prediction error $\risk$ is a maximizer of $G$.
\end{proposition}
\begin{proof}
    Let $\hat r=\risk$ for brevity.
    \boolfalse{expandrisk}
    The function $G$ satisfies $G(t) = M(t) - t\risk$ so the concavity of $M$ implies the concavity of $G$.
    We now show that $\risk$ is also a maximizer of $G$.
    By 1-strong concavity of $F$, we have for all $t\in\R$
    \begin{equation}
        F(\risk)\ge F(t) +(t-\risk)^2/2.
        \label{strong-concavity-of-F}
    \end{equation}
    For all $t\in\R$, thanks to \eqref{strong-concavity-of-F} we have
    \begin{align}
        G(\risk) - G(t)
        &= F(\risk) -\risk^2/2 - G(t), \\
        &\ge
        F(t) + (t-\risk)^2/2 -\risk^2/2 -G(t) = 0.
    \end{align}
    \booltrue{expandrisk}
\end{proof}

In the remaining of the present section, we assume that $h(\beta^*)<+\infty$. 

\begin{proposition}
    Assume that $h$ is convex and that $h(\beta^*)<+\infty$.
    Define the function $H:(0,+\infty)\rightarrow \R$ by
    \begin{equation}
        H(t)\coloneqq \sup_{\beta\in\R^p: \|\design(\beta-\beta^*)\| \le t } \frac{\eps^T\design(\beta-\beta^*) + h(\beta^*) - h(\beta)}{t}
        \label{def-H}
    \end{equation}
    for all $t>0$. Then the function $H$ is continuous and non-increasing on $(0,+\infty)$.
\end{proposition}
\begin{proof}
    As $h(\beta^*) < +\infty$, the function $F$ is concave and finite on $[0,+\infty)$.
    Thus $F$ is continuous on $(0,+\infty)$,
    and since $H(t) = (1/t)(F(t) + t^2/2 + h(\beta^*))$ the function $H$ is also continuous on $(0,+\infty)$.

    Let $s<t$ be two positive real numbers.
    For any $\beta\in\R^p$ such that $\|\design(\beta-\beta^*)\|\le t$, define $\tilde \beta = (s/t)\beta - (1-(s/t))\beta^*$.
    Then
    \begin{align}
        &\quad(s/t) [ \eps^T\design(\beta - \beta^*) + h(\beta^*) - h(\beta) ] \\
        &=
        \eps^T\design(\tilde\beta - \beta^*) + h(\beta^*) - (1-(s/t))h(\beta^*) - (s/t) h(\beta), \\
        &\le
        \eps^T\design(\tilde\beta - \beta^*) + h(\beta^*) - h(\tilde\beta), \\
        &\le
        s H(s),
    \end{align}
    where we used the convexity of $h$ for the first inequality
    and the fact that $\|\design(\tilde\beta-\beta^*)\|\le s$ for the second.
    By definition of the supremum, this implies $H(t) \le H(s)$.
\end{proof}

The functions $F, G$ and $H$ can be used to derive the following upper bound on the prediction error.

\begin{theorem}
    \label{thm:upper-bound-fixed-point-H}
    Assume that $h(\beta^*) < +\infty$.
    Then
    \begin{equation}
        \|\design(\hbeta - \beta^*) \|
        \le \inf\{ r>0: H(r)\le r \}.
        \label{eq:upper-bound-fixed-point-H}
    \end{equation}
\end{theorem}
\begin{proof}
    Let $\hat r=\risk$ for brevity.
    \boolfalse{expandrisk}
    Let $r>0$ be such that $H(r) \le r$. We have
    \begin{equation}
        G(r)+h(\beta^*) = r H(r) - r\risk \le r(r-\risk).
    \end{equation}
    By definition of $G$ we have $G(0) + h(\beta^*)\ge 0$.
    We prove that $r<\risk$ leads to a contradiction.
    Assume that $r<\risk$. Then we have $G(r) < -h(\beta^*) \le G(0)$. Since $\risk$ is a maximizer of the concave function $G$,
    this implies $0\le \risk < r$, hence a contradiction.
    Thus it must be the case that $r\ge \risk$ and the proof is complete.
    \booltrue{expandrisk}
\end{proof}

Notice that $H$ is nonnegative and non-increasing. Thus if $H$ is not equal to 0 everywhere on $(0,+\infty)$
then $H$ has a unique fixed-point. This fixed-point appears on the right hand side of  \eqref{eq:upper-bound-fixed-point-H}.

In summary, if the penalty $h$ is convex, we have established the following facts on the functions $F$, $G$ and $H$
defined in \eqref{def-F}, \eqref{def-G} 
and \eqref{def-H}.
\begin{enumerate}[label=(\roman*)]
    \item 
        The function $F$ is 1-strongly concave and the prediction error $\risk$ is the only maximizer of $F$.
    \item
        The function $G$ is concave and the prediction error $\risk$ is a maximizer of $G$.
    \item
        If $h(\beta^*) < +\infty$ then
        the function $H$ is continuous, non-increasing and the fixed-point of $H$ bounds
        the prediction error $\risk$ from above.
\end{enumerate}
Finally, note that $H$ and $G$ satisfy
\begin{equation}
    H(t) = \frac{G(t) + h(\beta^*)}{t} + \risk,
    \qquad
    \forall t>0
\end{equation}
provided that $h(\beta^*) < +\infty$.

\section{Optimistic lower-bounds}
\label{s:optimistic}

In this section, we show that the properties of $G$ and $H$
can be used to derive lower bounds on the prediction error $\risk$.
For instance, 
by concavity of $G$, if there exist two numbers $s<t$ such that
$G(s)< G(t)$ then any maximizer of $G$ is no smaller than $s$.
As the prediction error $\risk$ is a maximizer of $G$, this yields $s\le \risk$.

The following lower bound results hold for a given target vector $\beta^*$
and a given estimator $\hbeta$, namely, the penalized least-squares \eqref{hbeta}.
This contrasts with minimax lower-bounds that are derived from information theoretic
results such as Le Cam's Lemma or Fano's inequality.
Minimax lower bounds apply to any estimator and are \emph{pessimistic} in nature:
for any estimator $\hbeta$, a minimax lower bound yields
the existence of a worst-case target vector $\beta^*_{worst}$
for which the prediction error $\|\design(\hbeta - \beta^*_{worst}) \|$ is bounded from below.
Such minimax lower bounds are uninformative for target vectors that are not equal to $\beta^*_{worst}$.
The results of the present section are informative for any target vector, not only for a single worst-case target vector.
For this reason, the lower bounds of the present section are said to be \emph{optimistic}.

Recall that the function $H$ is non-increasing on $(0,+\infty)$. The first optimistic lower bound states that
$\lim_{t\rightarrow +\infty}H(t)$ bounds the prediction error from below.

\begin{theorem}
    \label{thm:liminf}
    Assume that $h(\beta^*) < +\infty$.
    Then for any $\eps$ we have
    \begin{equation}
        \lim_{t\rightarrow +\infty} H(t) = \inf_{t>0} H(t) \le \risk.
        \label{eq:liminf-lowerbound}
    \end{equation}
\end{theorem}
\begin{proof}
    Equality of the limit and the infimum is a consequence of the
    monotonicity of $H$.
    Since $\risk$ is a maximizer of $G$ we have $h(\beta^*) + G(\risk) \ge h(\beta^*) + G(t)$ for all $t> 0$, which can be rewritten as
    \begin{equation}
        H(t)\le \frac{G(\risk)+h(\beta^*)}{t} + \risk.
        \label{eq:Gplusbeta-star}
    \end{equation}
    Letting $t\rightarrow+\infty$ yields the desired inequality.
\end{proof}
As the function $H(\cdot)$ is non-increasing, a lower bound of the form
$H(t_0)\le \risk$ for some finite $t_0>0$ would be more appealing
than \eqref{eq:liminf-lowerbound}.
The next result shows that for a given small constant $\gamma$,
there exists a large enough $t_0>0$ such that $H(t_0)\le \risk + \gamma$.

\begin{theorem}
    \label{thm:t_0gamma}
    Assume that $h(\beta^*) < +\infty$ and let $t_0,\gamma>0$.
    If 
    \begin{equation}
        \eps^T\design(\hbeta-\beta^*) + h(\beta^*) - h(\hbeta)  - \|\design(\hbeta-\beta^*)\|^2 \le t_0 \gamma
        \label{eq:assum-t_0}
    \end{equation}
    then $H(t_0)\le \risk + \gamma$.
\end{theorem}
\begin{proof}
    Let $\hat r=\risk$ for brevity.
    \boolfalse{expandrisk}
    The choice $t=t_0$ in \eqref{eq:Gplusbeta-star} yields that
    \begin{equation}
        H(t_0)\le (1/t_0) (G(\risk)+h(\beta^*)) + \risk
        \le
        \gamma + \risk,
    \end{equation}
    since \eqref{eq:assum-t_0} can be rewritten as 
    $G(\risk)+h(\beta^*) \le t_0\gamma$.
    \booltrue{expandrisk}
\end{proof}
Typically, \Cref{thm:t_0gamma} is used with a constant $\gamma$ negligible compared to the prediction error $\risk$.

The next result provides a lower bound that mirrors the upper bound given in \Cref{thm:upper-bound-fixed-point-H}.
\Cref{thm:upper-bound-fixed-point-H} states that any fixed-point of $H$ bounds the prediction error $\risk$
from above.
For small $\alpha>0$, the quantity $(1-\alpha)r$ in \eqref{eq:assum-fixed-point-lower} below
can be interpreted as an ``almost fixed-point'' of $H$,
and such ``almost fixed-point'' of $H$ bounds the prediction error from below.
The following Theorem makes this precise.

\begin{theorem}
    \label{thm:almost-fixed-point-lower}
    Assume that $h(\beta^*) < +\infty$.
    Let $\alpha\in(0,1)$ and let $r>0$.
    If
    \begin{equation}
        H(  (1-\alpha) r  ) \le (1+\alpha^2) r
        \qquad
        \text{and}
        \qquad
        H( (1-\alpha^2)r ) \ge r
        \label{eq:assum-fixed-point-lower}
    \end{equation}
    then $\risk \ge (1-\alpha) r$.
\end{theorem}
\begin{proof}
    Let $\hat r=\risk$ for brevity.
    \boolfalse{expandrisk}
    Let $s = (1-\alpha)r$ and $t = (1-\alpha^2)r$ and note that $s < t$
    and that $t-s = (1-\alpha)\alpha r$.
    We prove that $\risk < s$ leads to a contradiction.
    Assume that $\risk < s$. By simple algebra,
    \begin{equation}
        G(s) - G(t) 
        =
        sH(s) - s\risk - tH(t) + t\risk.
    \end{equation}
    We have $\risk(t-s) < s(t-s)$ and \eqref{eq:assum-fixed-point-lower}
    can be rewritten as $H(s) \le (1+\alpha^2) r$ and $H(t) \ge r$.
    Combining these inequalities yields
    \begin{align}
        G(s) - G(t) 
        &\le
        (1+\alpha^2)s r - t r + s(t- s), \\
        &=
        r^2 [
            (1+\alpha^2)(1-\alpha)  - (1-\alpha^2) + (1-\alpha)^2 \alpha
        ].  
    \end{align}
    The bracket is equal to 0, so that $G(s) \le G(t)$ with $s<t$.
    By concavity of $G$, this implies that any maximizer of $G$ is no smaller than $s$.
    As $\risk$ is a maximizer of $G$, we have $s\le \risk$ which contradicts the assumption $\risk < s$.
    \booltrue{expandrisk}
\end{proof}

Finally, the following result 
will be useful to derive lower bounds when the penalty is too weak
compared to the noise random vector $\eps$.

\begin{theorem}
    \label{thm:V-sudakov}
    Let $h$ be a norm on $\R^p$.
    Then almost surely
    \begin{equation}
        \sup_{u\in \R^p: \|\design u\| \le 1} [
            \eps^T\design u - h(u)
        ] \le \|\design(\hbeta-\beta^*)\|.
    \end{equation}
\end{theorem}
\begin{proof}
    Let $t>0$ and define $s= t + \|\design\beta^*\|$. Let $u\in\R^p$ be such that $\|\design u \|\le 1$
    and let $\beta = t u$.
    Then $\|\design\beta\| \le t$ and $\|\design(\beta-\beta^*)\| \le s$ by the triangle inequality.
    Thus
    \begin{align}
        \eps^T\design u - h(u)
        &=
        (1/t)
        (\eps^T\design \beta - h(\beta)), \\
        &=
        (s/t)
        \frac{\eps^T\design(\beta - \beta^*)  + h(\beta^*) - h(\beta))}{s}
        - \frac{h(\beta^*) -\eps^T\design\beta^*}{t}, \\
        &\le (s/t) H(s) 
        - \frac{h(\beta^*) -\eps^T\design\beta^*}{t}.
    \end{align}
    As $t\to+\infty$, we obtain $s/t\to 1$ and
    $\eps^T\design u - h(u)
        \le
        \lim_{s\to +\infty} H(s)$.
    The definition of the supremum and \Cref{thm:liminf} completes the proof.
\end{proof}

It is not yet clear whether the above lower bound results are useful.
The following section will answer the following questions in the case where the penalty $h$ is proportional to the $\ell_1$-norm.
\begin{enumerate}[label=(\roman*)]
    \item 
        Each result of the present section relies on assumptions.
        Are these assumptions satisfied for specific examples of penalty $h$?
    \item
        How sharp are the above lower bounds?
        Are there examples of penalty $h$ such that the above lower bounds match
        known upper bounds?
        How large is the gap between \Cref{thm:upper-bound-fixed-point-H} and \Cref{thm:almost-fixed-point-lower}?
    \item The above results are deterministic: They hold for any realization of the noise random vector $\eps$.
        How to obtain lower bounds in expectation or in probability
        for a given noise distribution?
\end{enumerate}

\section{Application to Lasso}
\label{s:lasso}

The goal of this section is to use the method
of the previous section to provide novel insights
on the Lasso, that is, the estimator $\hbeta$ defined in
\eqref{hbeta}
with penalty 
\begin{equation}
    \label{h-lasso}
    h(\cdot) = \sqrt n \lambda \|\cdot\|_1,
\end{equation}
where $\lambda\ge 0$ is a tuning parameter.
The number of covariates $p$ is allowed to be larger than $n$.

The following notation will be needed.
Denote by $[p]$ the set $\{1,...,p\}$.
Let $(e_1,...,e_p)$ be the canonical basis in $\R^p$.
For any $T\subset[p]$, denote by $\Pi_T\in\R^{n\times n}$
the orthogonal projection onto the linear span of $\{\design e_j, j\in T\}$,
that is, onto the linear span of the columns of $\design$ with indices in $T$.
We say that a vector has sparsity $s$ if it has exactly $s$ nonzero components,
and for any $\beta\in\R^p$ we denote by $|\beta|_0$ the sparsity of $\beta$.

\subsection{On the compatibility constant}

For a subset $T\subset [p]$ and a constant $c_0\ge 1$, define the compatibility constant
\begin{equation}
    \phi(T,c_0) \coloneqq 
    \inf_{u\in\R^p: \|u_{T^c}\|_1 < c_0 \|u_T\|_1}
    \frac{\sqrt{|T|} \|\design u\| }{\sqrt n (\|u_T\|_1 - (1/c_0) \|u_{T^c} \|_1 )}.
\end{equation}

\begin{theorem}
    \label{thm:compatibility}
    Assume that the noise random vector is symmetric, i.e., 
    that $\eps$ and $-\eps$ have the same distribution.
    Let $\lambda\ge0$ be a tuning parameter,
    let $h$ be the penalty function \eqref{h-lasso}
    and let $T\subset [p]$.
    If $\phi(T,1) > 0$
    then there exists a target vector $\beta^*\in\R^p$ such that
    $\supp(\beta^*) \subset T$ and
    \begin{equation}
        \mathbb P\left(
        \frac{99}{100} \frac{\lambda \sqrt{|T|}}{\phi(T,1)} \le \risk
    \right) \ge 0.49.
    \end{equation}
\end{theorem}
\Cref{thm:compatibility} is a consequence of \Cref{thm:t_0gamma}.
The proof is given at the end of the present subsection.

Lower bounds on the prediction performance of Lasso estimators for ill-conditioned design have been derived
in \cite[Proposition 4]{dalalyan2017prediction} and in \cite{zhang2015optimal}.
These papers construct a specific design matrix $\design$ for which any Lasso estimator 
satisfy $\risk \ge \sigma n^{1/4}$.
The above lower bound holds for any design matrix and any support $T\subset[p]$.

For any constant $c>1$,
the Lasso satisfies 
$$\risk\lesssim c \lambda\sqrt s / \phi(T,c)$$
with high probability
provided that the tuning parameter $\lambda$ is large enough, see for instance
\cite{buhlmann2011statistics,dalalyan2017prediction}.
The above lower bound of \Cref{thm:compatibility} matches this upper bound, up to the gap $\frac{\phi(T,1)}{\phi(T,c)}$.

The constants $99/100$ and $0.49$ have been chosen arbitrarily.
It is clear from the proof below that $99/100$ can be replaced by a numerical constant arbitrarily close to 1,
and that $0.49$ can be replaced by a numerical constant arbitrarily close to $0.5$;
although the target vector $\beta^*$ depends on these numerical constants.

\begin{proof}[Proof of \Cref{thm:compatibility}]
    Let $q>0$ be a constant such that $\mathbb P( \|\eps\| \le q) \ge 0.99$
    and let $\Omega_1$ be the event $\{\| \eps \| \le q \}$.
    Define 
    \begin{equation}
        \gamma \coloneqq \lambda\sqrt{|T|}/(200\phi(T,1)),
        \qquad
        t_0\coloneqq (q+\lambda\sqrt T/\phi(T,1))^2/\gamma.
    \end{equation}
    By the definition of the infimum, there exists $u \in\R^p$
    such that
    \begin{equation}
        \|u_{T_c} \|_1 < \|u_T\|_1,
        \qquad
        \text{and}
        \qquad
        \frac{200}{199} \phi(T,1)  \ge 
        \frac{\sqrt{|T|} \|\design u\| }{\sqrt n (\|u_T\|_1 - \|u_{T^c} \|_1 )}.
    \end{equation}
    By homogeneity, we can assume that $\|\design u\| = 1$.
    We now define a target vector $\beta^*$ supported on $T$ by
    \begin{equation}
        \beta^*_T = - t_0 u_T,
        \qquad
        \beta^*_{T^c} = 0,
    \end{equation}
    so that $\|\beta^*\|_1 - \|\beta^*+t_0 u\|_1 = t_0(\|u_T\|_1 - \|u_{T^c} \|_1)$.

    By definition of $q$, on $\Omega_1$ we have
    \begin{align}
        &\qquad \eps^T\design(\hbeta - \beta^*) + h(\beta^*) - h(\hbeta)
        - \risk^2 \\
        &\le
        (q +\lambda\sqrt T / \phi(T,1) )^2 \risk - \risk^2, \\
        &\le (q+\lambda\sqrt T/\phi(T,1))^2 /4 \le t_0 \gamma,
    \end{align}
    where we used the elementary inequality $ab - a^2 \le b^2/4$.
    By \Cref{thm:t_0gamma},  the inequality
    $H(t_0) - \gamma \le \risk$ holds on $\Omega_1$.
    
    We now bound $H(t_0)$ from below on
    the event $\Omega_2 = \{ \eps^T\design(\beta-\beta^*) \ge 0 \}$.
    Let $\beta=\beta^* + t_0 u$.
    By construction, we have $\|\design(\beta-\beta^*)\| =t_0$
    and $\|\beta\|_1 = t_0 \|u_{T^c} \|_1$.
    By the definition of $H$, on $\Omega_2$ we have
    \begin{align}
        t_0 H(t_0)
        &\ge
        \eps^T\design(\beta-\beta^*)
        + h(\beta^*) - h(\beta), \\
        &\ge h(\beta^*) - h(\beta), \\
        &= \sqrt n \lambda t_0(\|u_T\|_1 - \|u_{T^c} \|_1) , \\
        &\ge t_0 (199/200) \lambda \sqrt{|T|}/\phi(T,1).
    \end{align}
    As the noise $\eps$ is symmetric, the event $\Omega_2$ has probability 1/2.

    By the union bound, the event $\Omega_1\cap\Omega_2$ has probability at least $0.49$
    and on this event we have
    \begin{equation}
        \risk \ge H(t_0) - \gamma
        \ge 
        \left(\frac{199}{200} - \frac{1}{200}\right)\lambda \sqrt{|T|}/\phi(T,1) = \frac{99}{100}\lambda \sqrt{|T|}/\phi(T,1) .
    \end{equation}
\end{proof}

\subsection{Tight upper and lower bounds of order $\lambda \sqrt s$ for well-conditioned design}

Certain conditions will be required on the design matrix in this section, namely, 
the Restricted Isometry Property (RIP)
introduced in \cite{candes2005decoding},
and the Restricted Eigenvalue (RE) condition
introduced in \cite{bickel2009simultaneous}.
For any $s=1,...,p$, define the constant $\delta_s\ge0$ as
the smallest $\delta\ge0$ such that
\begin{equation}
    (1-\delta) \|\beta\|
    \le (1/\sqrt n) \|\design\beta\|
    \le
    (1+\delta) \|\beta\|,
    \qquad
    \forall \beta\in\R^p \text{ such that } |\beta|_0\le s.
\end{equation}
We will say that the Restricted Isometry Property of order $s$ is satisfied,
or shortly that $RIP(s)$ holds,
if the constant $\delta_s$ is strictly less than 1.

Given a parameter $c_0>0$,
define the Restricted Eigenvalue constant $\kappa(c_0,s)$ by
\begin{equation}
    \kappa(c_0,s) \coloneqq 
    \inf_{\alpha\in\R^p: \sum_{j=s+1}^p \alpha_j^* \le c_0\sqrt s \|\alpha\| }
    \frac{\|\design\alpha\|}{\sqrt n \|\alpha\|},
\end{equation}
where $\alpha_1^*\ge...\ge\alpha_p^*$ is a non-decreasing rearrangement of $(|\alpha_1|,...,|\alpha_p|)$.
The Restricted Eigenvalue condition with parameters $c_0$ and $s$, or shortly $RE(c_0,s)$, is said to be satisfied if $\kappa(c_0,s) > 0$.
If both $RIP(s)$ and $RE(c_0,s)$ are satisfied then
\begin{equation}
    \label{comparison-re-rip-rip-bar}
    \kappa(c_0,s) \le (1-\delta_s) \le (1+\delta_s).
\end{equation}
For a fixed constant $\gamma>0$, define also the constants 
$c_0, \underline C$ and $\bar C$ by
\begin{align}
    \underline C &\coloneqq \frac{\sigma}{1+\delta_s},
    \qquad
    \qquad
    c_0 = \frac{1+\gamma + \sqrt 3}{\gamma},
    \label{def-c0}
    \\
    \bar C &\coloneqq 
    \frac{\sigma}{\kappa(c_0, s)}
    \bigg(
    1+
    \frac{\sigma \kappa(c_0,s) (
        \sqrt s
        +
        2\sqrt{\log3}
    )
    }{\lambda \sqrt s}
    +
    \frac{\sqrt 3}{\sqrt{\log(9ep/s)}}
    \bigg)
    .
\end{align}

\begin{theorem}
    \label{thm-exact-lasso}
    Assume that the noise random vector $\eps$ 
    has distribution $N(0,\sigma^2 I_{n\times n})$.
    Let $h$ be the penalty function \eqref{h-lasso}
    and let $s = |\beta^*|_0$.
    Let $\gamma > 0$ and define $c_0, \underline C, \bar C$ by \eqref{def-c0}.
    If the tuning parameter $\lambda$ satisfies
    \begin{equation}
        \label{lambda-tuning-log-9ep-s}
        \lambda \ge \sigma ( 1+\gamma) (1+\delta_s) (1+\sqrt{2\log(9ep/s)})
    \end{equation}
    then we have with probability at least $0.76$
    \begin{equation}
        \risk
        \le
        \bar C \lambda \sqrt s
        .
        \label{non-asymptotic-exact-upper}
    \end{equation}
    Furthermore, if $\bar C \le 2 \underline C$
    and if the components of $\beta^*$ satisfy
    \begin{equation}
        \min_{j:\beta_j^*\ne 0}|\beta_j^*| \ge \frac{(2\underline C -\bar C) \underline C \lambda}{\sqrt n},
        \label{explicit-beta-min} 
    \end{equation}
    then we have with probability at least $0.26$
    \begin{equation}
        \risk
        \ge
        \lambda \sqrt s \; \underline C \left(1-\sqrt{\bar C/\underline C - 1}\right).
        \label{non-asymptotic-exact-lower}
    \end{equation}
\end{theorem}
The proof is given at the end of the present section.
Upper bounds of the form
\eqref{non-asymptotic-exact-upper} have been obtained in \cite{bellec2016slope} with slightly worse constants.

The upper bound \eqref{non-asymptotic-exact-upper}
and the lower bound \eqref{non-asymptotic-exact-lower}
are tight in the following asymptotic regime.
Consider a sequence of problems indexed by $n$, so that 
$s,p,\beta^*,\design$ and $\lambda$ implicitly depend on $n$.
Next, consider an asymptotic regime with $p,n,s\rightarrow +\infty$ such that $s\log(p/s)/n\rightarrow 0$
and $p/s\to +\infty$,
whereas $\gamma$ and $c_0$ remain constant.
Next, set
\begin{equation}
    \lambda = \sigma (1+2\gamma) \sqrt{2\log(p/s)}.
    \label{eq:lambda-asymp}
\end{equation}
Assume that the rows of $\design$ are iid random vectors from a subgaussian and isotropic distribution.
Such assumption is satisfied, for instance, if the entries of $\design$ are iid $N(0,1)$ or Rademacher random variables.
Then it is known that
\begin{equation}
    \delta_s \to^{\mathbb P} 0,
\qquad
\kappa(c_0,s) \to^{\mathbb P} 1,
\end{equation}
where $\to^{\mathbb P}$ denotes the convergence in probability,
see for instance \cite{baraniuk2008simple,recht2010guaranteed,raskutti2010restricted,rudelson2013reconstruction}.
By definition of the constants $\bar C, \underline C$ in \eqref{def-c0},
this implies that $\bar C \to^{\mathbb P} 1$ and $\underline C \to^{\mathbb P} 1$.
Furthermore, \eqref{lambda-tuning-log-9ep-s} is satisfied with probability close to 1 for large enough $n,p,s$.
By \Cref{thm-exact-lasso}, there exist two constants $\underline c, \bar c$ that converge in probability to 1 such that,
for $n,p,s$ large enough we have
\begin{equation}
    \mathbb P\left(
        \underline c
        \le
        \frac{\risk}{\lambda \sqrt s}
        \le
        \bar c
    \right) \ge 0.25,
\end{equation}
provided that the nonzero components of $\beta^*$ are large enough so that
\eqref{explicit-beta-min}  is satisfied.
Thus, in the above asymptotic regime, the bounds of \Cref{thm-exact-lasso} are surprisingly tight:
The upper bound \eqref{non-asymptotic-exact-upper}
matches the lower bound \eqref{non-asymptotic-exact-lower}
on an event of constant probability.
The exact asymptotic rate is known to be $\sqrt{2s\log(p/s)}$, cf. \cite{su2016slope}.
Thus the prediction error of the Lasso with tuning parameter \eqref{eq:lambda-asymp}
achieves the exact asymptotic rate, up the constant $1+2\gamma$.
The Lasso with tuning parameter \eqref{eq:lambda-asymp} not only achieves the asymptotic
rate $\sqrt{2s\log(p/s)}$ for the prediction error,
but also achieves the asymptotic constant $(1+2\gamma)\sqrt 2$.
As the constant $\gamma>0$ can be chosen arbitrarily small, $(1+\gamma)\sqrt 2$ can be made arbitrarily close to $\sqrt 2$
which is the optimal asymptotic constant (\cite{su2016slope}).

The condition \eqref{explicit-beta-min} requires that the nonzero coefficients of the target vector $\beta^*$
are detectable. If $\lambda$ is chosen as in \eqref{eq:lambda-asymp} then
the nonzero coefficients of $\beta^*$ should be larger than $\sigma \sqrt{\log(p/s)/n}$, up to a multiplicative constant.
If $\lambda$ is chosen to be of order $\sigma\sqrt{\log(p)}$, then \eqref{explicit-beta-min} requires
that the nonzero coefficients of $\beta^*$ are larger than $\sigma\sqrt{\log(p)/n}$ up to a multiplicative constant.

The proof of \Cref{thm-exact-lasso} given below relies on \Cref{thm:upper-bound-fixed-point-H} for the upper bound
and \Cref{thm:almost-fixed-point-lower} for the lower bound.
Thus the present subsection illustrates a situation
where the ratio between
the upper bound of 
\Cref{thm:upper-bound-fixed-point-H} and the lower bound of \Cref{thm:almost-fixed-point-lower}
converges to 1.

\vspace{0.1in}
\begin{proof}[Proof of \Cref{thm-exact-lasso}]
    \Cref{thm-exact-lasso} has two claims.
    The first claim, \eqref{non-asymptotic-exact-upper}, is an upper bound on the prediction error $\risk$
    while the second claim, \eqref{non-asymptotic-exact-lower}, is a lower bound.
    The first claim is a consequence of the following proposition.

    \begin{proposition}
        \label{prop:exact-lasso-upper-H}
        Under the assumptions of 
        \Cref{thm-exact-lasso}, there exists an event $\Omega$
        of probability at least $0.76$ such that on $\Omega$ we have
        \begin{equation}
            H(t) \le \bar C \lambda \sqrt s,
            \qquad
            \forall t \ge 0,
        \end{equation}
        provided that the tuning parameter $\lambda$ satisfies \eqref{lambda-tuning-log-9ep-s}.
    \end{proposition}
    Proposition~\ref{prop:exact-lasso-upper-H} is proved in \Cref{s:proof-prop-exact-lasso-upper-H}.
    By \Cref{thm:upper-bound-fixed-point-H},
    Proposition~\ref{prop:exact-lasso-upper-H} readily implies
    the first claim of \Cref{thm-exact-lasso}.

    In order to prove the second claim of \Cref{thm-exact-lasso},
    we first derive the following lower bound on $H$.
    \begin{proposition}
        \label{prop:exact-lasso-lower-H}
        Assume that $\eps$ has a symmetric distribution, i.e., that
        $\eps$ and $-\eps$ have the same distribution.
        Let $h(\cdot)$ be the penalty function \eqref{h-lasso},
        let $s = |\beta^*|_0$ and assume that $s\ge1$.
        Let $t > 0$.
        Assume the nonzero coefficients of $\beta^*$ satisfy
        \begin{equation}
            \min_{j:\beta_j^*\ne 0}|\beta_j^*| \ge \frac{t}{(1+\delta_s) \sqrt s \sqrt n}.
            \label{beta-min}
        \end{equation}
        Then we have $H(t)\ge\lambda\sqrt s / (1+\delta_s) = \underline C \lambda \sqrt s$
        with probability at least $0.5$.
    \end{proposition}
    The proof of \Cref{prop:exact-lasso-lower-H} is given
    in \Cref{s:proof-prop-exact-lasso-lower-H}.
    We are now ready to combine \Cref{prop:exact-lasso-lower-H}
    and \Cref{prop:exact-lasso-upper-H}
    and complete the proof of \Cref{thm-exact-lasso}.

    First, notice that $\bar C \ge 1 \ge \underline C$.
    Define $\alpha\coloneqq\sqrt{\bar C/\underline C - 1}$
    and $r\coloneqq \underline C \lambda \sqrt s$.
    By simple algebra,
    $(1+\alpha^2) r  = \bar C \lambda \sqrt s$,
    so that on the event of Proposition~\ref{prop:exact-lasso-upper-H}
    we have 
    \begin{equation}
        \label{condition-1}
        H(r(1-\alpha)) \le (1+\alpha^2) r.
    \end{equation}
    We now apply Proposition~\ref{prop:exact-lasso-lower-H}
    to 
    \begin{equation}
        t=(1-\alpha^2)r  = ( 2 \underline C -\bar C)\lambda \sqrt s.
    \end{equation}
    Then \eqref{beta-min} is equivalent to \eqref{explicit-beta-min}
    and by Proposition~\ref{prop:exact-lasso-lower-H}
    we have 
    \begin{equation}
         H(t) = H((1-\alpha^2)r) \ge \underline C \lambda \sqrt s = r
         \label{condition-2}
    \end{equation}
    with probability at least $0.5$.
    By the union bound, there exists an event of probability at least $0.26$ on which both \eqref{condition-1} and \eqref{condition-2} hold.
    \Cref{thm:almost-fixed-point-lower} completes the proof.
\end{proof}

\subsection{On the Lasso with small tuning parameter}

The previous section shows that if the tuning parameter of the Lasso is of order $\sigma\sqrt{\log(p/s)}$
where $s$ is the sparsity of the target vector,
then the prediction error of the Lasso is no smaller than $\sqrt s \lambda$.

The following result shows that if the tuning parameter of the Lasso is slightly smaller
than $\sigma\sqrt{\log(p/s)}$, then the prediction error becomes substantially larger than $\sqrt s \lambda$.

\begin{theorem}
    \label{thm:lasso-d-log-p-d}
    Assume that $\eps\sim N(0,\sigma^2I_{n\times n})$.
    Let $h$ be the penalty function \eqref{h-lasso}.
    Let $d\ge 1$.
    If the tuning parameter satisfies
    \begin{equation}
        \lambda \le \frac{1-\delta_{2d}}8 \sigma\sqrt{\log(p/(5d))},
        \label{eq:lambda-too-small-log-p-5d}
    \end{equation}
    then we have
    \begin{equation}
        \frac{1-\delta_{2d}}{8(1+\delta_d)} \sigma\sqrt{d \log(p/(5d))}
        \le
        \E \risk 
        \label{eq:lasso-d-log-p-d-lower}
    \end{equation}
\end{theorem}

The above result makes no sparsity assumption on the target vector $\beta^*$.
To understand the implication of \Cref{thm:lasso-d-log-p-d},
assume in this paragraph that the vector vector $\beta^*$ has sparsity $s \llless d$.
Then the optimal rate for the prediction error is of order $\sqrt{s\log(p/s)}$.
As explained in the previous section, this rate is achieved, for instance, by the Lasso with tuning parameter of order $\sigma\sqrt{\log(p/s)}$.
The above result says that if the tuning parameter is too small in the sense of  \eqref{eq:lambda-too-small-log-p-5d},
i.e., $\lambda \lesssim \sigma\sqrt{\log(p/d)}$,
then the prediction error of the Lasso is at least of order $\sqrt{d\log(p/d)}$.
Even though the size of the true model is $s$,
the Lasso with small tuning parameter (as in \eqref{eq:lambda-too-small-log-p-5d})
suffers a prediction error of order at least $\sqrt{d\log(p/d)}$
which is the optimal prediction error when the true model is of size $d$ with $d\ggg s$.
A result similar to \eqref{eq:lasso-d-log-p-d-lower} was obtained in \cite[Proposition 14]{sun2013sparse} in a random design setting
where the design has iid $N(0,1)$ entries.
Theorem 7.1 in \cite{lounici2011oracle} yields a lower bound on the prediction performance of Lasso of the form
$\risk \ge |\hbeta|_0^{1/2} \lambda /(2\phi_{max})$ where $\phi_{max}$ is the maximal eigenvalue of $\frac 1 n \design^T\design$,
and this result proposes conditions under which $|\hbeta|_0\ge|\beta^*|_0$ holds with high probability.

\vspace{0.1in}

\begin{proof}[ Proof of \Cref{thm:lasso-d-log-p-d} ]
    Taking expectations in \Cref{thm:V-sudakov}, we obtain
    \begin{equation}
        \E \sup_{u\in V: \|\design u\| \le 1} [
            \eps^T\design u - h(u)
        ] \le  \E \|\design(\hbeta-\beta^*)\|.
    \end{equation}
    Let $\Omega \subset\{0,1\}^p$ be given by Lemma~\ref{lemma:extraction}.
    Define $b = \frac{1}{(1+\delta_{d}) \sqrt d \sqrt n}$.
    For any $w\in\Omega$, define $u_w$ as $u_w = b w$.
    Then,
    thanks to the properties of $\Omega$ in Lemma~\ref{lemma:extraction},
    \begin{equation}
        \|\design u_w \|
        =
        b \|\design w \|
        \le b  (1+\delta_{d}) \sqrt{n}\sqrt d
        = 1
    \end{equation}
    by definition of $b$.
    Next, notice that $h(u_w) = \lambda \sqrt d /(1+\delta_d)$ for all $w\in\Omega$.
    Thus
    \begin{equation}
        \E \sup_{u\in V: \|\design u\| \le 1} [
            \eps^T\design u - h(u)
        ] \ge
        \E \sup_{w\in \Omega}
            \eps^T\design u_w
            - \frac{\lambda \sqrt d }{ 1+\delta_d}.
            \label{to-be-combined-1}
    \end{equation}
    For any two distinct $w,w'\in\Omega$, by Lemma~\ref{lemma:extraction} we have
    $\E[(\eps^T\design(w -w'))^2] \ge \sigma^2 n b^2 (1-\delta_{2d})^2 d$.
    By Sudakov's lower bound (see for instance Theorem 13.4 in \cite{boucheron2013concentration})
    we get
    \begin{align}
        \E[\sup_{w\in\Omega}\eps^T\design u_w]
        &\ge 
        \sigma (1/2) \sqrt n  (1-\delta_{2d}) b \sqrt{d} \sqrt{\log|\Omega|},   \\
        &\ge
        \sigma
        \sqrt{n}
        (1/4) (1-\delta_{2d})  b d \sqrt{\log(p/(5d))}
        = \frac{\sigma (1-\delta_{2d})}{4(1+\delta_d)} \sqrt{d\log(p/(5d))}.
        \label{to-be-combined-2}
    \end{align}
    Combining \eqref{to-be-combined-1} and the previous display,
    we obtain the desired lower bound provided that 
    $\lambda$ satisfies \eqref{eq:lambda-too-small-log-p-5d}.
\end{proof}

\section{Gaussian noise and the integrated counterpart of $F$}
\label{s:gaussian}

Results of \Cref{s:variational-characterization} hold for any realization of the noise 
vector $\eps$, without any assumption on its probability distribution.
In this section, we assume that $\eps$ has normal distribution
$N(0, \sigma^2 I_{n \times n} )$
where $I_{n \times n}$ is the identity matrix of size $n\times n$ and $\sigma>0$ is the noise level.
Furthermore, we assume that the infimum
\begin{equation}
    \label{t_c}
    t_c \coloneqq \inf_{\beta\in\R^p: h(\beta) < +\infty} \|\design(\beta-\beta^*)\|
\end{equation}
is attained at some $\beta_0\in\dom h$, where $\dom h$ is the effective domain of $h$
defined by 
$\dom h \coloneqq \{x\in\R^p: h(x) < +\infty \}$.
Next, following the strategy of \cite{chatterjee2014new}, define the function 
$f:[0,+\infty)\rightarrow [-\infty,+\infty)$ by
\begin{equation}
    f(t) \coloneqq
    \E[F(t)].
    \label{def-f}
\end{equation}
where $F(\cdot)$ is the random function defined in \eqref{def-F}
and the expectation is taken with respect to $\eps\sim N(0,\sigma^2I_{n\times n})$.
We have established in \Cref{s:variational-characterization} that for any realization of the noise vector $\eps$,
the function $F(\cdot)$ is 1-strongly concave. By integration, this readily implies that $f$
is also 1-strongly concave.
Furthermore, as the penalty function is nonnegative, we have by the Cauchy-Schwarz inequality
\begin{equation}
    f(t) \le \E \|\eps\| t - t^2/2,
\end{equation}
so that $f(t) \to -\infty$ as $t\to +\infty$.
These observations yield the existence of a unique maximizer $t_f$ of $f$.
We gather these results on the function $f$ in the following Theorem.

\begin{theorem}
    \label{thm:concavity}
    Let $\beta^*\in\R^p$, let $h$ be a convex penalty function 
    and let $t_c\ge0$ be defined in \eqref{t_c}.
    Assume that the infimum \eqref{t_c} is attained.
    Then $f(t) = -\infty$ for $t<t_c$,
    $f(t)$ is finite for $t\ge t_c$,
    and $f$ is 1-strongly concave on $[t_c,+\infty)$.
    Thus
    the function $f$ has a unique maximizer $t_f$ and
    for all $t\ge t_c$ we have
    \begin{equation}
        \label{eq:}
        f(t_f) \ge f(t) + (t-t_f)^2/2.
    \end{equation}
\end{theorem}
The influential paper of \cite{chatterjee2014new} provided a concentration result of $\risk$
around the maximizer $t_f$ in shape constrained models, i.e., for penalty functions that are indicator functions
of closed convex sets.
\cite{bellec2016bounds} established the following concentration bounds of the prediction error $\risk$
around its median and its mean.
\begin{proposition}[\cite{bellec2016bounds}]
    \label{prop:lipschitz-ell}
    Assume that $\eps\sim(0,\sigma^2I_{n\times n})$.
    Assume that the penalty function $h$ satisfies \Cref{assum-h}.
    Then the function $\eps\rightarrow \risk$
    is a 1-Lipschitz function of the noise random vector $\eps$.
    Thus, for any $x>0$ we have
    \begin{align}
        \mathbb P(\risk \ge m + \sigma x) &\le \mathbb P(N(0,1) \ge x), \\
        \mathbb P(\risk \le m - \sigma x) &\le \mathbb P(N(0,1) \le -x),
        \label{eq:above-concentration-ineq1}
    \end{align}
    where $m$ is the median of the random variable $\risk$.
\end{proposition}
The fact that $f$ is Lipschitz is proved in \cite{bellec2016bounds}.
Then, the above concentration inequalities are direct consequence
of the Gaussian concentration Theorem \cite[Theorem 10.17]{boucheron2013concentration}.
The fact that $\eps\to\risk$ is 1-Lipscthitz also yields
that the median of $\risk$ and its expectation are equal up to an additive constant, i.e., we have
\begin{equation}
    \Big| m - \E[\risk] \Big| \le \sigma \sqrt{\pi/2},
    \label{eq:ineq-median-mean}
\end{equation}
cf. the discussion after equation (1.6) in \cite[page 21]{ledoux2013probability}.

It is possible to recast the concentration inequalities \eqref{eq:above-concentration-ineq1}
using \emph{stochastic dominance}. Indeed, the above concentration inequalities yield
\begin{equation}
    \label{eq:coupling}
    \mathbb P
    (
        |\risk - m|
    >
    \sigma
    x
    )
    \le
    \mathbb P(|Z|>x),
\end{equation}
for some $Z\sim N(0,1)$.
By coupling and stochastic dominance (here, $|Z|$ dominates $|\risk - m|/\sigma$), there exists a large enough probability space $\Omega$
such that $Z$ and $\risk$ are both random variables on $\Omega$ and such that
\begin{equation}
    |
        \risk
        -
        m
    |
    \le
    \sigma |Z|.
\end{equation}
holds almost surely on $\Omega$
(see for instance Theorem 7.1 in \cite{hollander2012probability}).

The next result sheds light on the relationship between the maximizer $t_f$ of the integrated function $f(\cdot)$
and the mean or median of the prediction error $\risk$.
In short, the absolute error between $(t_f)^{1/2}$ and the median $m^{1/2}$ 
is no more than a constant.
The same holds for the absolute error between $(t_f)^{1/2}$ and $\E[\risk]^{1/2}$.

\begin{theorem}
    \label{thm:main}
    Let $\beta^*\in\R^p$. Let $h$ be a penalty function satisfying Assumption~\ref{assum-h}
    and assume that the infimum \eqref{t_c} is attained.
    If $t_f$ is the unique maximizer of $f$, then
    we have
    \begin{align}
        \Big|
        \sqrt{t_f}
        -
        \sqrt m
        \Big|
        &\le
        \sqrt{21\sigma/2}
        \le
        3.25\sqrt{\sigma}
        ,
        \label{eq:claim-median}\\
        \qquad
        \Big|
        \sqrt{t_f}
        -
        \sqrt{\E[\risk ]}
        \Big|
        &\le
        \sqrt{\sigma}
        (
        \sqrt{21/2}
        +
        (\pi/2)^{1/4}
        )
        \le
        4.40\sqrt{\sigma}
        ,\label{eq:claim-expectation}
    \end{align}
    where $m$ is the median of the random variable $\risk$.
\end{theorem}
The proof of \Cref{thm:main} is given at the end of the current section.

Combining \Cref{thm:main}, the discussion above \eqref{eq:coupling} and some algebra, we obtain the following inequalities.
Let $Z\sim N(0,1)$ and 
let $\hat r=\risk$ for brevity.
\boolfalse{expandrisk}
If $\Omega$ is the rich enough probability space on which \eqref{eq:coupling}
holds almost surely, 
then we have almost surely
    \begin{align}
        |
        \sqrt{\risk}
        -
        \sqrt{t_f}
        |
        &\le
        \sqrt\sigma
        (
            \sqrt{21} 
            + 
            \sqrt{|Z|}
        ),
        \label{eq1-cor}
        \\
        |
        {\risk}
        -
        {t_f}
        |
        &\le 2 \sqrt{21\sigma t_f} + \sigma(21 + |Z|),
        \label{eq2-cor-moment-1-Z}
        \\
        | \risk^2 - t_f^2 |
        &\le
        c (\sqrt \sigma t_f^{3/2} + \sigma^2 + Z^2),
        \label{eq4-cor-moment-2-Z}
        \\
        |\E [ \risk ] - t_f |
        &\le 2\sqrt{21\sigma t_f} + 22\sigma,
        \label{eq3-cor-moment-1-Esp} \\
        | \E [ \risk^2] - t_f^2 |
        &\le
        c (\sqrt \sigma t_f^{3/2} + 2 \sigma^2),
        \label{eq5-cor-moment-2-Esp}
    \end{align}
    where $c>0$ is an absolute constant.
\begin{proof}
    By the elementary inequality $|\sqrt a-\sqrt b|\le \sqrt{|a-b|}$, inequality \eqref{eq:coupling} implies that
    $|\sqrt{m} - \sqrt{\risk}| \le \sqrt{\sigma|Z|}.$
    Thus inequality \eqref{eq1-cor} is a consequence of  \Cref{thm:main} and the triangle inequality.
    Let $C=\sqrt{21\sigma}$. For the second inequality, \Cref{thm:main} readily implies
    \begin{equation}
        |t_f - m| = |\sqrt{t_f} - \sqrt{m}| (\sqrt{t_f} + \sqrt{m})
        \le C(2 \sqrt{t_f} + C).
    \end{equation}
    Combining this with \eqref{eq:coupling} and the triangle inequality completes the proof of \eqref{eq2-cor-moment-1-Z}.
    A similar argument can be used to prove \eqref{eq4-cor-moment-2-Z} from \eqref{eq2-cor-moment-1-Z}. 
    Jensen's inequality and \eqref{eq2-cor-moment-1-Z} imply \eqref{eq3-cor-moment-1-Esp},
    and finally
    Jensen's inequality and \eqref{eq4-cor-moment-2-Z} imply \eqref{eq5-cor-moment-2-Esp}.
    \booltrue{expandrisk}
\end{proof}

By \Cref{thm:main}, the median and the mean of the prediction error $\risk$ are both
close to $t_f$.
Thus upper and lower bounds on $\risk$ can be obtained
from upper and lower bounds on $t_f$. This follows the strategy outlined in \cite{chatterjee2014new} in shape restricted regression.
The results of the present section show that if the noise random vector $\eps$
has standard normal distribution, then the concentration results initially obtained in shape restricted regression in \cite{chatterjee2014new}
also hold for penalized least-squares estimators in linear regression.

Finally, let us derive a simple condition to obtain an upper bound on $t_f$.

\begin{theorem}
    If $\E[y] = \design\beta^*$ and for some $s>0$ we have
    $f(s) + h(\beta^*)\le s^2$, or equivalently
    \begin{equation}
        \E\sup_{\beta\in\R^p: \|\design(\beta-\beta^*)\|\le s}\left[\eps^T\design(\beta-\beta^*) + h(\beta^*) - h(\beta) \right]
        \le s^2,
    \end{equation}
    then $t_f \le s$.
\end{theorem}
\begin{proof}
    The assumption implies $f(s) \le f(0) + s^2$ since $-h(\beta^*) \le f(0)$.
    As $f$ is 1-strongly concave, if $d$ is a supergradient of $f$ at $s$,
    then we have
    $f(0) \le f(s) + d(0 -s) - s^2$, which implies $ds \le 0$.
    Hence, $f$ is non-increasing at $s$.
    By concavity, this implies that $t_f$, the maximum of $t$ belongs to $[0, s]$.
\end{proof}
If $h$ is the indicator function of a closed convex set $K$ and $\design$ is the identity matrix,
then Proposition 1.3 in \citet{chatterjee2014new} shows that 
$t_f\le s$ is granted provided that
\begin{equation}
    \E\sup_{\beta\in K: \|\beta-\beta^*\|\le s}\left[\eps^T(\beta-\beta^*) \right]
    \le s^2 /2.
\end{equation}
The above result improves upon Proposition 1.3 in
\cite{chatterjee2014new} by a factor 1/2.

\vspace{0.5in}

\begin{proof}[Proof of \Cref{thm:main}]
    Let $\hat r=\risk$ for brevity.
    \boolfalse{expandrisk}
    First, let us prove that for any fixed $t\ge t_c$,
    the function $\eps\rightarrow F(t)$ is $t$-Lipschitz.
    Let $e_1,e_2\in\R^n$
    and let $$F_i =
        \sup_{\beta\in\R^p: \|\design(\beta-\beta^*)\| \le t}
        \left(
        e_i^T\design(\beta-\beta^*)
        - h(\beta)
        \right)
    $$ for $i=1,2$.
    To prove that $\eps\to F(t)$ is a $t$-Lipschitz function of $\eps$,
    it is enough to prove that $F_1 - F_2 \le t\|e_1 - e_2\|$.
    For any $\beta\in\R^p$ such that $\|\design(\beta-\beta^*)\| \le t$ and $h(\beta) <+\infty$, we have
    \begin{equation}
        e_1^T\design(\beta-\beta^*)
        - h(\beta)
        =
        e_2^T\design(\beta-\beta^*)
        - h(\beta)
        +
        (e_1 - e_2)^T\design(\beta-\beta^*) 
        \le F_2 + t\|e_1 - e_2\|.
    \end{equation}
    By definition of the supremum, this proves that $F_1 \le F_2 + t\|e_2 - e_2\|$.
    We have established that
    the function $\eps\rightarrow F(t)$ is $t$-Lipschitz.

    The concentration of a Lipschitz function of a standard normal random variable \cite[Theorem 5.6]{boucheron2013concentration} yields
    that for any $x\ge 0$ and any fixed $t\ge t_c$ we have
    \begin{equation}
        \mathbb P(
            F(t)
            > f(t) + \sigma t x
        ) \le e^{-x^2/2},
        \qquad
        \mathbb P(
            f(t)
            > F(t) + \sigma t x
        ) \le e^{-x^2/2}.
        \label{eq:above-concentration-ineq2}
    \end{equation}
    Let $\tau>0$ be a numerical constant that will be specified later.
    On the event $\mathcal A \coloneqq \{ \risk \le m \}$,
    by monotonicity of the supremum we have
    \begin{equation}
        F(m) \ge F(\risk) + (\risk^2 - m^2)/2. 
    \end{equation}
    Define the event $\mathcal B\coloneqq \{ \risk \ge m - \tau \sigma \}$.
    On $\mathcal A\cap\mathcal B$ we have
    \begin{equation}
        (\risk^2 - m^2)/2
        = (1/2) (\risk - m)(\risk + m)
        \ge - (1/2) \tau\sigma (\risk + m) 
        \ge - m\tau\sigma.
    \end{equation}
    Inequality $F(\risk) \ge F(t_f)$ holds almost surely since $\risk$ is a maximizer of $F$.
    Next, define the event
    $\mathcal C \coloneqq \{ F(t_f) \ge f(t_f) - \tau\sigma t_f \}$.
    On $\mathcal A\cap\mathcal B\cap \mathcal C$ we have
    \begin{equation}
        F(m)
        \ge
        f(t_f) - \tau\sigma(t_f + m).
    \end{equation}
    By \Cref{thm:concavity} and the strong concavity of the function $f(\cdot)$, we have $f(t_f) \ge f(m) + (m-t_f)^2/2$.
    Finally, define the event
    $\mathcal D \coloneqq \{ f(m) \ge F(m) - m \tau\sigma \}$.
    On $\mathcal A\cap\mathcal B\cap\mathcal C\cap \mathcal D$ we obtain
    \begin{equation}
        F(m)
        \ge F(m) - \tau\sigma(2 m+t_f) + (m-t_f)^2/2,
    \end{equation}
    which implies $|m-t_f| \le \sqrt{\tau\sigma(6 \max(m,t_f))}$.
    By simple Algebra,
    \begin{equation}
        |\sqrt m - \sqrt{t_f} |
        = \frac{|m-t_f|}{\sqrt m + \sqrt{t_f}}
        \le 
        \frac{|m-t_f|}{\sqrt{\max(t_f,m) }}
        \le
        \sqrt{6\tau\sigma}.
    \end{equation}
    Let $\tau = 7/4$ so that the right hand side of the previous display is equal to $\sqrt{21\sigma/2}$.
    This inequality holds on the event 
    $\mathcal A\cap\mathcal B\cap\mathcal C\cap\mathcal D$.
    To complete the proof of  \eqref{eq:claim-median}, it remains to show that this event has positive probability.
    By definition of the median, $\mathbb P(\mathcal A) \ge 1/2$.
    Using the union bound and the above concentration inequalities
    \eqref{eq:above-concentration-ineq1}-\eqref{eq:above-concentration-ineq2}
    for the events $\mathcal B, \mathcal C$ and $\mathcal D$,
    for $\tau=7/4$ we have
    \begin{align}
        \mathbb P(\mathcal A\cap\mathcal B\cap\mathcal C\cap \mathcal D)
        &\ge 
        1
        - \mathbb P(\mathcal A^c)
        - \mathbb P(\mathcal B^c)
        - \mathbb P(\mathcal C^c)
        - \mathbb P(\mathcal D^c), \\
        &\ge
        1-
        1/2
        - 2 e^{-\tau^2/2}
        - \mathbb P( N(0,1) > \tau)
        > 0.01.
    \end{align}
    We now prove \eqref{eq:claim-expectation}. By the elementary inequality $|\sqrt a-\sqrt b|\le\sqrt{|a-b|}$,
    inequality \eqref{eq:ineq-median-mean} yields that
    $|\sqrt m - \E[\risk]^{1/2}| \le (\pi/2)^{1/4}$.
    The triangle inequality completes the proof of \eqref{eq:claim-expectation}.
    \booltrue{expandrisk}
\end{proof}


\bibliographystyle{plainnat}
\bibliography{../../bibliography/db}

\appendix

\section{Varshamov-Gilbert extraction Lemma}

\begin{lemma}[Lemma 2.5 in \cite{giraud2014introduction}]
    \label{lemma:extraction}
    For any positive integer $d$ less than $p/5$, there exists a subset $\Omega$ of the set
    $\{w \in\{0,1\}^p: |w|_0 = d \}$ that fulfills 
    \begin{align}
        \log (|\Omega|) &\ge  (d/2)  \log\left(\frac{p}{5d}\right),
        \qquad
        \sum_{j=1}^p \mathbf 1_{w_j \ne w_j'} = \|w - w'\|^2 > d, 
    \end{align}
    for any two distinct elements $w$ and $w'$ of $\Omega$,
    where $|\Omega|$ denotes the cardinal of $\Omega$.
\end{lemma}

\section{Preliminaries for the proof of \Cref{thm-exact-lasso}}
\label{s:proof-RE-exact}

\subsection{A stochastic Lemma}

\begin{lemma}
    \label{prop-klar}
    Let $\eps\sim N(0,\sigma I_{n\times n})$.
    For any $S\subset[p]$, let $\Pi_S\in\R^{n\times n}$ be the orthogonal projection onto
    the linear span of the columns of $\design$ indexed in $S$.
    Let $T\subset[p]$ and $s=|T|$.
    The events
    \begin{align}
        \Omega_1 &\coloneqq
        \left\{
            \|\Pi_T \eps\| \le \sigma(\sqrt s + 2\sqrt{\log(3)}
        \right\}, \\
        \Omega_2 &\coloneqq
        \left\{
            \max_{j=1,...,p} \left[
                \frac{
                    \max_{S:|S|=j}\|\Pi_S(I_{n\times n}-\Pi_T) \eps\|
                    }{
                    \sqrt j(1+ \sqrt{2\log(9ep/j)}) 
                    }
            \right] \le 1
        \right\}, \\
        \Omega & \coloneqq \Omega_1\cap\Omega_2
    \end{align}
    satisfy $\mathbb P(\Omega_1) \ge 0.888$, $\mathbb P(\Omega_2)\ge 0.875$ 
    and $\mathbb P(\Omega) \ge 0.76$.
    Furthermore, define
    \begin{equation}
        g_j \coloneqq (1/\sqrt n) \eps^T (I_{n\times n}-\Pi_T) \design e_j,
        \qquad 
        j=1,...,p
        \label{def-g_j}
    \end{equation}
    and let $g_1^*\ge...\ge g_p^*$ be a nondecreasing rearrangement of $(|g_1|,...,|g_p|)$.
    Define also $\mu_1,...,\mu_p$ by
    \begin{equation}
        \label{def-muj}
        \mu_j \coloneqq
        \sigma (1+\delta_j)
        (1+ \sqrt{2\log(9ep/j)}),
        \qquad
        j=1,...,s.
    \end{equation}
    Then, on $\Omega_2$ we have
    \begin{equation}
        \label{integral}
        \max_{j=1,...,s}\left[\frac{g_j^*}{\mu_j}\right]\le 1,
        \qquad
        \sum_{j=1}^s (g_j^* - \mu_{s})_+^2 \le \mu_s^2 s \frac{3}{\log(9ep/s)}.
    \end{equation}
\end{lemma}
\begin{proof}
    Without loss of generality, assume that $\sigma = 1$.
    The random variable $\|\Pi_T\eps\|^2$ is a $\chi^2$ random variable with at most $s$ degrees of freedom.
    The bound $\|\Pi_T\eps\|\le \sqrt s +\sqrt{2x}$ holds for any $x>0$ with probability $1-e^{-x}$.
    The choice $x=\log(9)$ grants $\mathbb P(\Omega_1)\ge 1 - 1/9 \ge 0.888$.

    We now bound $\mathbb P(\Omega_2)$ from below.
    Let $I=I_{n\times n}$ for brevity.
    For any $j=1,...,p$  and any $S\subset[p]$ with $|S|=j$,
    the function
    $\eps\rightarrow\|\Pi_S(I-\Pi_T) \eps\|$ is a 1-Lipschitz function of $\eps$
    and its expectation satisfies
    $\E \|\Pi_S(I-\Pi_T) \eps\|
    \le \sqrt{\E \|\Pi_S(I-\Pi_T) \eps\|^2}
    =\sqrt{\Tr(\Pi_S(I-\Pi_T)\Pi_S)} \le \sqrt{|S|}$
    since the matrix inside the trace has rank at most $|S|$ and operator norm at most 1.
    By the concentration of Lipschitz functions of Gaussian random variables 
    (see, for instance, Theorem 5.6 in \cite{boucheron2013concentration})
    we have $\|\Pi_S(I-\Pi_T) \eps\|\le \sqrt{|S|} + \sqrt{2x}$ with probability at least $1-e^{-x}$.
    Let $x=j\log 9$.
    By the union bound, we have with probability at least $1-e^{-x} = 1-9^{-j}$
    \begin{multline}
        \max_{S:|S|=j}\|\Pi_S(I-\Pi_T) \eps\|
        \le \sqrt j + \sqrt{2\left(x + \log{j \choose p } \right)}, \\
        \le \sqrt j + \sqrt{2(x+ j\log\frac{ep}{j})}
        = \sqrt j + \sqrt{2j\log\frac{9ep}{j}}.
    \end{multline}
    Again, using the union bound over $j=1,...,p$ we obtain that 
    $\mathbb P(\Omega_2) \ge 1-\sum_{j=1}^p 9^{-j} \ge 0.875$.

    To prove the left inequality of \eqref{integral}, observe that on $\Omega_2$ we have
    \begin{align}
        (g_j^*)^2 
        \le \frac{1}{j} \sum_{k=1}^j (g_k^*)^2
        &=
        \max_{v\in\R^p: \|v\|=1, |v|_0=j}
        \frac{ (\eps^T(I-\Pi_T)\design v)^2 }{nj}, \\
        &\le
        (1+\delta_j)^2 (1+\sqrt{2\log(9ep/j)})^2 = \mu_j^2.
    \end{align}
    It remains to prove the right inequality of \eqref{integral}. On $\Omega_2$,
    for any $j\le s$, using that $-\sqrt{\log(9ep/s)}\le -\sqrt{\log(9ep/(j+s))}$
    we obtain
    \begin{align}
        g_j^* - \mu_{s}
        \le \mu_j - \mu_{s}
        &\le \sqrt 2 (1+\delta_{s})(\sqrt{\log(9ep/j)} - \sqrt{\log(9ep/(j+s))}),\\
        &= \sqrt 2 (1+\delta_{s})\frac{\log(1+s/j)}{\sqrt{\log(9ep/j)}+\sqrt{\log(9ep/(j+s))}}, \\
        &\le \frac{\sqrt 2 (1+\delta_{s})}{\sqrt{\log(9ep/s)}} \log(1+s/j).
    \end{align} 
    To complete the proof of \eqref{integral}, we use the following identity from \cite[(A.17)]{su2016slope}:
    \begin{equation}
        \sum_{j=1}^s
        \log(1+s/j)^2
        \le s \int_0^1 \log(1+1/x)^2dx
        \le 3s.
    \end{equation}
\end{proof}

\subsection{Upper bound on $H(\cdot)$}
\label{s:proof-prop-exact-lasso-upper-H}

\begin{proof}[Proof of Proposition~\ref{prop:exact-lasso-upper-H}]
    Let $\Omega$ be the event from Proposition~\ref{prop-klar} above.
    The following argument is deterministic conditionally on $\Omega$.
    
    Let $\mu_s$ be defined in \eqref{def-muj} with $j=s$
    and observe that $\lambda \ge (1+\gamma) \mu_s$.
    Let $\beta\in\R^p$ and $u=\beta-\beta^*$.
    Let $T\subset [p]$ be the support of $\beta^*$.
    Let $\Pi_T\in\R^{n\times n}$ be the orthogonal projection onto linear span
    of the columns of $\design$ indexed in $T\coloneqq \supp(\beta^*)$.
    By simple algebra, we have almost surely
    \begin{align}
        \eps^T\design u 
        =
        \eps^T\Pi_T \design u
        + \eps^T(I_{n\times n}-\Pi_T) \design u
        &=
        \eps^T\Pi_T \design u
        + \eps^T(I_{n\times n}-\Pi_T) \design u_{T^c}, \\
        &\le
        \|\Pi_T \eps\| \|\design u\|
        + \sqrt n \sum_{j \in T^c} g_j u_j,
    \end{align} 
    where  $g_1,...,g_p$ are defined in \eqref{def-g_j}.
    Let $\Omega=\Omega_1\cap\Omega_2$ be the event
    of probability at least $1/2$
    defined in Lemma~\ref{prop-klar}. 
    On $\Omega$, we have
    $
            \|\Pi_T\eps\|
            \le \sigma(\sqrt s + 2 \sqrt{\log 3})
    $.

    Let $\hat T\subset \{1,...p\}$ be the set of the $s$ indices $j$ with largest $|g_j|$
    and define $\hat S \coloneqq \hat T \cap T^c$.
    By construction we have $|\hat S| \le s$, $\hat S \cap \supp(\beta^*) = \varnothing$  (or equivalently $\hat S\subset T^c$)
    and for any $j\notin \hat S\cup T$ we have $|g_j|\le g_s^*$
    where $g_1^*\ge...\ge g_p^*$ is a nondecreasing rearrangement of $(|g_1|,...,|g_p|)$.
    Thus
    \begin{equation}
        \sum_{j \in T^c} g_j u_j
        =
        \sum_{j \in \hat S} g_j u_j
        +
        \sum_{j \notin \hat S \cup T} g_j u_j
        \le
        \sum_{j\in \hat S} |g_j| | u_j|
        +
        g_s^*\sum_{j\notin \hat S \cup T} |u_j|.
    \end{equation}
    By the triangle inequality and Cauchy-Schwarz inequality, it is clear that
    $\|\beta^*\|_1 - \|\beta\|_1 \le \|u\|_2\sqrt s - \sum_{j\notin T} |u_i|.$
    Combining the above inequalities, we obtain that
    \begin{align}
        \eps^T\design u + h(\beta^*) - h(\beta)
        \le \|\Pi_T\eps\|\|\design u\|
        +
        \sqrt n
        \left[
            \sum_{j\in \hat S} |g_j| | u_j|
            +
            g_s^*\sum_{j\notin \hat S\cup T}|u_j|
            +
            \lambda \sqrt s \|u \|
            -
            \lambda\sum_{j\notin T}|u_j|
        \right].
    \end{align}
    On $\Omega$ we have $g_s^* \le \mu_s \le \lambda/(1+\gamma)$.
    Furthermore, on $\Omega$, the bracket of the previous display is bounded from above by
    \begin{equation}
        \sum_{j\in\hat S}
        (|g_j| - \mu_s)_+ |u_j|
        +
        \lambda\sqrt s \|u\|
        -
        \gamma\mu_s \sum_{j\in T^c} |u_j|.
    \end{equation}
    Let $(\mu_j)_{j=1,...p}$ be defined in \eqref{def-muj}.
    By the Cauchy-Schwarz inequality and \eqref{integral} from Lemma~\ref{prop-klar},  
    on $\Omega$ we have
    \begin{multline}
        \sum_{j\in \hat S}
        (|g_j| - \mu_s)_+ |u_j|
        \le
        (g_j^* - \mu_s)_+ u_j^*
        \sum_{j=1}^s
        (g_j^* - \mu_s)_+ u_j^*,\\
        \le \|u\| \Big(\sum_{j=1}^s (g_j^* - \mu_{s})_+^2 \Big)^{1/2}
        \le \|u\| \mu_s \sqrt{s} \eta,
    \end{multline}
    where $\eta \coloneqq \sqrt{3/\log(9ep/s)}$.
    In summary, we have established that on $\Omega$, for any $\gamma\ge0$,
    \begin{multline}
        \eps^T\design u + h(\beta^*) - h(\beta)
        \le 
        \sigma(\sqrt s + 2 \sqrt{\log 3})
        \|\design u\| \\
         + \mu_s\sqrt n
        \Big[
            \sqrt s \|u\| (1+\gamma+\eta)
            -
            \gamma \sum_{j\notin \hat S \cup T} |u_j|
        \Big].
        \label{previous}
    \end{multline}
    Observe that if $\gamma >0$ then the constant $c_0$ is finite.
    On the one hand, if the bracket is positive, then by the $RE(c_0,s)$ with $c_0$ defined in \eqref{def-c0}
    we obtain $\sqrt n \|u\| \le \|\design u\| / \kappa(c_0,s) $ and thus
    \begin{equation}
        \eps^T\design(\beta-\beta^*)
        +
        \lambda\sqrt n(\|\beta^*\|_1 - \|\beta\|_1)
        \le
            \sigma(\sqrt s + 2 \sqrt{\log 3})
    \|\design u\|
            + \mu_s \sqrt s \frac{1+\gamma+\eta}{\kappa(c_0,s)} \|\design u\|.
        \label{previous2}
    \end{equation}
    On the other hand, inequality \eqref{previous2} holds trivially if the bracket of \eqref{previous} is  negative.
    We have proved that on $\Omega$,
    \begin{equation}
        H(t)
        \le 
        \sigma\sqrt s + \sigma2\sqrt{\log 3}
        + \mu_s \sqrt s (1+\gamma+\eta) /\kappa(c_0,s)
        = \bar C \lambda \sqrt s,
    \end{equation}
    where $\eta = \sqrt{3/\log(9ep/s)}$.
\end{proof}

\subsection{Lower bound on $H(\cdot)$}
\label{s:proof-prop-exact-lasso-lower-H}

\begin{proof}[Proof of Proposition~\ref{prop:exact-lasso-lower-H}]
Let $b>0$ be equal to the right hand side of \eqref{beta-min}.
Define $\beta\in\R^p$ by $\beta_j = \beta_j^* - b\text{sign}(\beta_j^*)$
if $\beta_j^*\ne 0$ and $\beta_j = 0$ otherwise.
By construction, the vector $\beta$ has the same support and signs as $\beta^*$.
Furthermore, $RIP(s)$ grants
\begin{equation}
    \|\design(\beta-\beta^*)\|
    \le (1+\delta_s)  \sqrt n \|\beta-\beta^*\|
    \le (1+\delta_s)  b \sqrt n \sqrt s
    = t.
\end{equation}
Since $ \|\design(\beta-\beta^*)\|\le t$, 
by definition of $H(\cdot)$,
on the event $\{ \eps^T\design(\beta-\beta^*)\ge 0 \}$ we have
\begin{align}
    H(t)
    &\ge  (1/t) \eps^T\design(\beta-\beta^*)+ (1/t)\sqrt n \lambda(\|\beta^*\|_1 - \|\beta\|_1), \\
    &\ge \sqrt n \lambda s b / t
    =  \lambda\sqrt s / (1+\delta_s).
\end{align}
Finally, note that since $\eps$ has a symmetric distribution,
the event $\{ \eps^T\design(\beta-\beta^*)\ge 0 \}$ has probability at least $0.5$.
\end{proof}

\end{document}